\documentclass[11pt,reqno]{article}

\usepackage[usenames]{color}
\usepackage{amssymb}
\usepackage{amsmath}
\usepackage{amsthm}
\usepackage{amsfonts}
\usepackage{amscd}
\usepackage{graphicx}

\usepackage[colorlinks=true,
linkcolor=webgreen,
filecolor=webbrown,
citecolor=webgreen]{hyperref}

\definecolor{webgreen}{rgb}{0,.5,0}
\definecolor{webbrown}{rgb}{.6,0,0}
\usepackage{color}
\usepackage{fullpage}
\usepackage{float}

\usepackage{cases}

\usepackage{graphics,amsmath,amssymb}
\usepackage{amsthm}
\usepackage{amsfonts}
\usepackage{latexsym}

\begin{document}

\theoremstyle{plain}
\newtheorem{theorem}{Theorem}
\newtheorem{remark}{Remark}
\newtheorem{lemma}{Lemma}
\newtheorem{example}{Example}
\newtheorem{definition}{Definition}
\newtheorem{proposition}{Proposition}
\newtheorem{corollary}{Corollary}

\begin{center}
\vskip 1cm{\Large\bf Sums of powers of integers via differentiation}
\vskip .2in \large Jos\'{e} Luis Cereceda \\
{\normalsize Collado Villalba, 28400 (Madrid), Spain} \\
\href{mailto:jl.cereceda@movistar.es}{\normalsize{\tt jl.cereceda@movistar.es}}
\end{center}

\begin{abstract}
For integer $k \geq 0$, let $S_k$ denote the sum of the $k$th powers of the first $n$ positive integers $1^k + 2^k + \cdots + n^k$. For any given $k$, the power sum $S_k$ can in principle be determined by differentiating $k$ times (with respect to $x$) the associated exponential generating function $\sum_{k=0}^{\infty}S_k x^k/k!$, and then taking the limit of the resulting differentiated function as $x$ approaches $0$. In this paper, we exploit this method to establish a couple of seemingly novel recurrence relations, one of them involving the even-indexed power sums $S_2, S_4,\ldots, S_{2k}$, and the other the odd-indexed power sums $S_{1}, S_3, \ldots, S_{2k-1}$, with both recurrence relations depending explicitly on the parameter $N = n + \frac{1}{2}$. From this, we obtain a determinantal formula of order $k$ which yields $S_{2k}$ [$S_{2k-1}$] in the Faulhaber form, that is, as an odd [even] polynomial in $N$. As a byproduct, we discover a new determinantal formula for the Bernoulli number $B_{2k}$. Furthermore, we show that $S_{2k}$ and $S_{2k-1}$ can be obtained by taking the corresponding higher order derivatives of the Chebyshev polynomials of the second kind.
\end{abstract}

\section{Introduction}

For integers $k \geq 0$ and $n \geq 1$, consider the sum of powers $S_k = 1^k + 2^k + \cdots + n^k$. For fixed $n$, the sequence $S_0, S_1, \ldots, S_k, \ldots\, $ has the exponential generating function
\begin{equation}\label{egf}
\sum_{k=0}^{\infty} S_k \frac{x^k}{k!} = \sum_{r=1}^{n} e^{rx}.
\end{equation}
By summing the geometric series on the right-hand side of \eqref{egf}, one can in principle determine $S_k$, for any given $k$, as the following limit when $x$ approaches $0$:
\begin{equation}\label{met1}
S_k = \lim_{x\to 0} \, \frac{d^k}{d x^k} \left( \frac{e^{(n+1)x} -e^x}{e^x -1} \right), \quad k \geq 0,
\end{equation}
where $\frac{d^k}{d x^k} f(x)$ denotes the $k$th derivative of $f(x)$ with respect to $x$. For example, for $k=1$, we have
\begin{align*}
S_1 & = \lim_{x\to 0} \, \frac{d}{d x} \left( \frac{e^{(n+1)x} -e^x}{e^x -1} \right) \\
& =  \lim_{x\to 0} \, \frac{e^x - (n+1) e^{(n+1)x} + n e^{(n+2)x}}{( e^x - 1)^2}.
\end{align*}
In order to evaluate the involved limit, we apply L'H\^{o}pital's rule twice to get
\begin{align*}
\lim_{x\to 0} \, \frac{e^x - (n+1) e^{(n+1)x} + n e^{(n+2)x}}{( e^x - 1)^2}
& = \lim_{x\to 0} \, \frac{1 - (n+1)^2 e^{nx} + n(n+2) e^{(n+1)x}}{2( e^x - 1)} \\
& = \lim_{x\to 0} \, \frac{1}{2} \big( - n(n+1)^2 e^{(n-1)x} + n(n+1)(n+2) e^{nx} \big) \\
& = \frac{1}{2}n(n+1).
\end{align*}

This procedure was described by Schaumberger in \cite{shaum}. Moreover, in a very recent paper \cite{miller}, Due\~{n}ez, Hamakiotes, and Miller have developed a related approach to power sums through differentiation. Instead of using $\sum_{r=1}^{n} e^{rx}$, the authors of \cite{miller} use the geometric series
\begin{equation*}
1 + x + x^2 + \cdots + x^n = \frac{1-x^{n+1}}{1-x}.
\end{equation*}
To find $S_k$, one has to apply the operator $x\frac{d}{dx}$ $k$-times to each side of the previous equation and then take the limit as $x$ approaches $1$. This can be expressed compactly as
\begin{equation}\label{met2}
S_k = \lim_{x\to 1} \, \left( x\frac{d}{d x}\right)^k \left( \frac{1 - x^{n+1}}{1-x} \right), \quad k \geq 1.
\end{equation}
As pointed out in \cite{miller}, $k$ successive applications of $x\frac{d}{dx}$ to the geometric series produces an expression of the form (in our notation): $Q_n(x;k)/(1-x)^{k+1}$, where $Q_n(x;k)$ turns out to be a polynomial in $x$ of degree $n+k+1$. As a result, formula \eqref{met2} can be equivalently expressed as (cf.~\cite[Equation (37)]{gaut1})
\begin{equation}\label{miller2}
S_k = \frac{(-1)^{k+1}}{(k+1)!} \, \lim_{x\to 1} \frac{d^{k+1}}{d x^{k+1}} Q_n(x;k), \quad k \geq 1.
\end{equation}
Regarding the polynomial $Q_n(x;k)$, from Equations (10) and (13) of \cite{hsu} it can be deduced that
\begin{equation*}
Q_n(x;k) = \sum_{j=0}^k j! \genfrac{\{}{\}}{0pt}{}{k}{j}\bigg[ x^j (1-x)^{k-j} - x^{n+1} \sum_{r=0}^j
\binom{n+1}{j-r} x^r (1-x)^{k-r} \bigg],
\end{equation*}
where $\genfrac{\{}{\}}{0pt}{}{k}{j}$ are the Stirling numbers of the second kind. Alternatively, $Q_n(x;k)$ can be written in terms of the Eulerian polynomials $A_k(x)$ as follows:
\begin{equation*}
Q_n(x;k) = A_k(x) - x^{n+1} \sum_{j=0}^k \binom{k}{j} (n+1)^{j}(1-x)^{j} A_{k-j}(x),
\end{equation*}
with $A_0(x) =1$, $A_j(x) = \sum_{i=1}^j \genfrac{<}{>}{0pt}{}{j}{i} x^i$ for $j \geq 1$, and where $\genfrac{<}{>}{0pt}{}{j}{i}$ are the Eulerian numbers. As an example, we have
\begin{equation*}
Q_n(x;2) = x + x^2- (n+1)^2 x^{n+1} + (2n^2 +2n-1)x^{n+2} - n^2 x^{n+3},
\end{equation*}
and then, from \eqref{miller2}, it follows that
\begin{equation*}
S_2 = -\frac{1}{6} \, \lim_{x\to 1} \frac{d^{3}}{d x^{3}} Q_n(x;2) = \frac{1}{6}n(n+1)(2n+1).
\end{equation*}

Additionally, it is pertinent to mention that the action of $\big(x\frac{d}{dx} \big)^{k}$ on an arbitrary function $f(x)$ possessing $k$th derivatives amounts to the following operation on $f(x)$ (see e.g. \cite[Equation (3)]{knopf}):
\begin{equation*}
\left( x\frac{d}{dx} \right)^{k} \! f(x) = \sum_{j=1}^k \genfrac{\{}{\}}{0pt}{}{k}{j} x^j \frac{d^j f}{d x^j}.
\end{equation*}
Hence, one can equally compute $S_k$ by means of the formula
\begin{equation}\label{knopf}
S_k = \sum_{j=1}^k \genfrac{\{}{\}}{0pt}{}{k}{j} \lim_{x\to 1} \frac{d^j}{d x^j} \left(
\frac{1 - x^{n+1}}{1-x} \right), \quad k\geq 1.
\end{equation}
For example, for $k=3$, from \eqref{knopf} one obtains
\begin{align*}
S_3 & = \lim_{x\to 1} \frac{d}{d x} \left( \frac{1 - x^{n+1}}{1-x} \right)
+ 3 \lim_{x\to 1} \frac{d^2}{d x^2} \left( \frac{1 - x^{n+1}}{1-x} \right)
+ \lim_{x\to 1} \frac{d^3}{d x^3} \left( \frac{1 - x^{n+1}}{1-x} \right) \\
& = \frac{1}{2}n(n+1) + n (n^2-1) + \frac{1}{4}n( n^3 -2n^2 -n +2) \\
& = \frac{1}{4}n^2(n+1)^2.
\end{align*}

It should be noted that the methods presented in \cite{miller} and \cite{shaum} for obtaining $S_k$ are entirely equivalent to each other. This follows from the fact that, whenever $k \geq 1$, the right-hand side of \eqref{met1} can be converted into the right-hand side of \eqref{met2} (and vice versa) by interchanging $e^x \leftrightarrow x$, $\frac{d^k}{d x^k} \leftrightarrow \big( x\frac{d}{dx} \big)^k$, and $\lim_{x\to 0} \leftrightarrow \lim_{x\to 1}$. Furthermore, both formulas \eqref{met1} and \eqref{met2} provide $S_k$ as an explicit polynomial in $n$ of degree $k+1$. Unfortunately, however, the differentiation method embodied in either of the two fundamental formulas \eqref{met1} or \eqref{met2} is rather cumbersome to handle, and it becomes impractical for determining $S_k$ manually when $k$ increases above $k=3$ or $k=4$. Something similar can be said of the formulas \eqref{miller2} and \eqref{knopf}. In this paper, we turn this situation around and take advantage of the differentiation formula \eqref{met1} to derive a couple of seemingly novel recurrence relations, one of them involving the $k$ even-indexed power sums $S_2, S_4,\ldots, S_{2k}$, and the other the $k$ odd-indexed power sums $S_{1}, S_3, \ldots, S_{2k-1}$ (where $k \geq 1$), with both recurrences depending explicitly on the parameter $N = n + \frac{1}{2}$.

To this end, we first notice that, by making the transformation $x \to ix$, where $i$ is the imaginary unit $i = \sqrt{-1}$ and $x$ is assumed to be a real variable, the left-hand and the right-hand side of \eqref{egf} split into a real part plus an imaginary part, namely
\begin{equation*}
\sum_{k=0}^{\infty} (-1)^k S_{2k} \frac{x^{2k}}{(2k)!} + i \, \sum_{k=0}^{\infty} (-1)^k
S_{2k+1} \frac{x^{2k+1}}{(2k+1)!} = \sum_{r=1}^{n} \cos rx + i \, \sum_{r=1}^{n} \sin rx ,
\end{equation*}
from which it follows that
\begin{align}
& \sum_{k=0}^{\infty} (-1)^k S_{2k} \frac{x^{2k}}{(2k)!} = \sum_{r=1}^{n} \cos rx, \label{rpart} \\[-2mm]
\intertext{and}
& \sum_{k=0}^{\infty} (-1)^k S_{2k+1} \frac{x^{2k+1}}{(2k+1)!} = \sum_{r=1}^{n} \sin rx. \label{ipart}
\end{align}

\begin{remark}
When $n =1$, from \eqref{rpart} and \eqref{ipart} we retrieve the Maclaurin series expansion of the functions $\cos x$ and $\sin x$, namely, $\cos x = \sum_{k=0}^{\infty} (-1)^k \frac{x^{2k}}{(2k)!}$ and $\sin x = \sum_{k=0}^{\infty} (-1)^k \frac{x^{2k+1}}{(2k+1)!}$.
\end{remark}

Denoting $f_n(x) \equiv \sum_{r=1}^{n} \cos rx$ and $g_n(x) \equiv \sum_{r=1}^{n} \sin rx$, we can see from \eqref{rpart} and \eqref{ipart} that the following relations hold for any integer $k \geq 0$:
\begin{equation}\label{der1}
\left.
\begin{array}{l}
D_x^{2k}f_n(x)\big|_{x=0} = (-1)^k S_{2k}, \\[2mm]
D_x^{2k+1}f_n(x)\big|_{x=0} = 0,
\end{array}\right\}
\end{equation}
and
\begin{equation}\label{der2}
\left.
\begin{array}{l}
D_x^{2k}g_n(x)\big|_{x=0} = 0, \\[2mm]
D_x^{2k+1}g_n(x)\big|_{x=0} = (-1)^k S_{2k+1},
\end{array}\right\}
\end{equation}
where, for example, $D_x^{2k}f_n(x)\big|_{x=0}$ means the $2k$th derivative of $f_n(x)$ with respect to $x$, $\frac{d^{2k}}{d x^{2k}}f_n(x)$, evaluated at $x=0$.

Based on relations \eqref{der1} and \eqref{der2}, in Section 2, we establish the above-mentioned couple of recurrence relations  for $S_{2k}$ and $S_{2k-1}$ (Theorem \ref{th:1}). Then, from these recurrences, and using Cramer's rule, we obtain respective determinantal formulas of order $k$ for $S_{2k}$ and $S_{2k-1}$ (equations \eqref{det1} and \eqref{det2}). This yields $S_{2k}$ [$S_{2k-1}$] in the form of an odd [even] polynomial in $N =n +\frac{1}{2}$, in accordance with Faulhaber's theorem on sums of powers of integers (see e.g. \cite{beardon,hersh,knuth}). Alternatively, these determinantal formulas can be expressed in terms of $S_1 = \frac{1}{2}n(n+1)$, as shown in equations \eqref{det3} and \eqref{det4}. As a byproduct, we discover a new determinantal formula for the Bernoulli number $B_{2k}$ (equation \eqref{detber}). Furthermore, in Section 3, we show that $S_{2k}$ and $S_{2k-1}$ can be obtained by taking the corresponding higher order derivatives of the Chebyshev polynomials of the second kind. Using this fact and Fa\`{a} di Bruno's formula for the $k$th derivative of a composite function, we provide an alternate formula for the power sums $S_{2k}$ and $S_{2k-1}$ (Theorems \ref{th:2} and \ref{th:3}). We conclude in Section 4 with several relevant comments.

\section{Recurrent and determinantal formulas for $S_{2k}$ and $S_{2k-1}$}

In order to prove Theorem \ref{th:1} below, we need to invoke the following pair of well-known, elementary trigonometric identities.

\begin{lemma}
For $N = n + \frac{1}{2}$, we have
\begin{align}
\sum_{r=1}^{n} \cos rx & = \frac{\sin Nx - \sin(x/2)}{2\sin(x/2)}, \label{id1} \\[-2mm]
\intertext{and}
\sum_{r=1}^{n} \sin rx & = \frac{\cos(x/2) - \cos Nx}{2\sin(x/2)}. \label{id2}
\end{align}
\end{lemma}
\begin{proof}
Both \eqref{id1} and \eqref{id2} quickly follow by summing the geometric series $\sum_{r=1}^{n} e^{rx}$ with $x$ replaced by $ix$, that is
\begin{align*}
\sum_{r=1}^{n} e^{i rx} & = \frac{e^{i(n+1)x} - e^{ix}}{e^{ix} -1} = \frac{e^{iNx} - e^{ix/2}}{e^{ix/2} - e^{-ix/2}} \\
& = \frac{\cos Nx + i\sin Nx -\cos(x/2) - i\sin(x/2)}{2i\sin(x/2)},
\end{align*}
and then identifying the real (imaginary) part of the last equation with $\sum_{r=1}^{n} \cos rx$ (respectively, $\sum_{r=1}^{n} \sin rx$).
\end{proof}

Equipped with relations \eqref{der1} and \eqref{der2}, and identities \eqref{id1} and \eqref{id2}, we proceed to prove the following theorem.

\begin{theorem}\label{th:1}
Let $N = n + \frac{1}{2}$. Then, for any integer $k \geq 1$, we have
\begin{align}
\sum_{j=1}^k 4^j \binom{2k+1}{2j} S_{2j} & = 4^k N^{2k+1} - N, \label{th1} \\[-2mm]
\intertext{and}
\sum_{j=1}^k 4^j \binom{2k}{2j-1} S_{2j-1} & = 4^k N^{2k} - 1. \label{th2}
\end{align}
\end{theorem}
\begin{proof}
To prove \eqref{th1}, we first write \eqref{id1} in the form
\begin{equation*}
2 \sin(x/2) F_n(x) = \sin Nx,
\end{equation*}
where $F_n(x) = f_n(x) + \frac{1}{2}$. It is easily seen that, for any integer $j \geq 0$,
\begin{equation}\label{pf1}
\left.
\begin{array}{l}
D_x^{2j}\sin Nx\big|_{x=0} = 0, \\[2mm]
D_x^{2j+1}\sin Nx\big|_{x=0} = (-1)^{j} N^{2j+1},
\end{array}\right\}
\end{equation}
and thus it follows that
\begin{equation}\label{pf2}
2 D_x^{2k+1} \big(\sin(x/2) F_n(x) \big)\big|_{x=0} = (-1)^{k} N^{2k+1}.
\end{equation}
Regarding the left-hand side of the previous equation, we apply Leibniz's formula for the $k$th derivative of the product of two functions to obtain
\begin{align*}
2D_x^{2k+1} \big(\sin(x/2) F_n(x) \big)\big|_{x=0} & = 2 \sum_{j=0}^{2k+1} \binom{2k+1}{j} D_x^{2k+1-j}\sin(x/2)\big|_{x=0}
D_x^j F_n(x)\big|_{x=0} \\
& = (-1)^{k}\frac{N}{4^{k}} + 2\sum_{j=1}^{2k+1} \binom{2k+1}{j} D_x^{2k+1-j}\sin(x/2)\big|_{x=0}
D_x^{j} f_n(x)\big|_{x=0}.
\end{align*}
Because of relations \eqref{der1} and \eqref{pf1}, this can be expressed as
\begin{align}
2D_x^{2k+1} \big(\sin(x/2) F_n(x) \big)\big|_{x=0} & = (-1)^{k}\frac{N}{4^{k}} + 2\sum_{j=1}^{k} \binom{2k+1}{2j} D_x^{2k+1-2j}\sin(x/2)\big|_{x=0} D_x^{2j} f_n(x)\big|_{x=0}  \notag \\
& = (-1)^{k}\frac{N}{4^{k}} + (-1)^{k} \frac{1}{4^{k}} \sum_{j=1}^k 4^j \binom{2k+1}{2j} S_{2j}. \label{pf3}
\end{align}
Hence, from \eqref{pf2} and \eqref{pf3}, we get \eqref{th1}.

Analogously, to prove \eqref{th2}, we first write \eqref{id2} as
\begin{equation*}
2 \sin(x/2) g_n(x) = \cos(x/2) - \cos Nx.
\end{equation*}
Thus, noting that
\begin{equation*}
\left.
\begin{array}{l}
D_x^{2j}\cos Nx\big|_{x=0} = (-1)^k N^{2j}, \\[2mm]
D_x^{2j+1}\cos Nx\big|_{x=0} = 0,
\end{array}\right\}
\end{equation*}
which holds for any integer $j \geq 0$, it follows that
\begin{equation}\label{pf4}
2 D_x^{2k} \big(\sin(x/2) g_n(x) \big)\big|_{x=0} = (-1)^{k} \Big( \frac{1}{4^k} - N^{2k} \Big).
\end{equation}
Using Leibniz's formula, the left-hand side of the preceding equation can be written as
\begin{equation*}
2 D_x^{2k} \big(\sin(x/2) g_n(x) \big)\big|_{x=0} = 2\sum_{j=0}^{2k} \binom{2k}{j} D_x^{2k-j}\sin(x/2)\big|_{x=0}
D_x^{j} g_n(x)\big|_{x=0},
\end{equation*}
which, by virtue of relations \eqref{der2} and \eqref{pf1}, can in turn be expressed as
\begin{align}
2 D_x^{2k} \big(\sin(x/2) g_n(x) \big)\big|_{x=0} & = 2\sum_{j=1}^{k} \binom{2k}{2j-1} D_x^{2k+1-2j}\sin(x/2)\big|_{x=0}
D_x^{2j-1} g_n(x)\big|_{x=0} \notag \\
& = (-1)^{k-1} \frac{1}{4^k} \sum_{j=1}^{k} 4^j \binom{2k}{2j-1} S_{2j-1}. \label{pf5}
\end{align}
Thus, from \eqref{pf4} and \eqref{pf5}, we get \eqref{th2}.
\end{proof}

\begin{remark}
In view of \eqref{th1} and \eqref{th2}, one immediately obtains the identity
\begin{equation*}
\sum_{j=1}^k 4^j \binom{2k+1}{2j} S_{2j} = \Big( n +\frac{1}{2} \Big)
\sum_{j=1}^k 4^j \binom{2k}{2j-1} S_{2j-1}, \quad k \geq 1.
\end{equation*}
For example, letting $k =4$ in the above identity leads to
\begin{equation*}
\frac{9S_2 + 126 S_4 + 336 S_6 + 144S_8}{S_1 + 28 S_3 + 112S_5 +64S_7} = 2n+1,
\end{equation*}
which holds for any integer $n \geq 1$.
\end{remark}

\begin{remark}
In \cite{chen2} (see also \cite{acu} for related work), the authors derived the following couple of recurrence relations:
\begin{align*}
\sum_{j=1}^k \binom{2k+1}{2j} S_{2j} & = \frac{1}{2} \big[ (n+1)^{2k+1} + n^{2k+1} -2n-1 \big], \\[-2mm]
\intertext{and}
\sum_{j=1}^k \binom{2k}{2j-1} S_{2j-1} & = \frac{1}{2} \big[ (n+1)^{2k} + n^{2k}-1 \big],
\end{align*}
which should be compared with the recurrences \eqref{th1} and \eqref{th2}.
\end{remark}

The first $k$ instances of the recurrence relation \eqref{th1} provides the following linear system of $k$ equations in the unknowns $S_2, S_4,\ldots, S_{2k}$:
\begin{equation*}
\begin{aligned}
& \quad\,\, 4 \binom{3}{2} S_2 = 4N^3 -N, \\
& \quad\,\, 4 \binom{5}{2} S_2 + 4^2 \binom{5}{4} S_4 = 4^2 N^5 -N,  \\
& \quad\,\, 4 \binom{7}{2} S_2 + 4^2 \binom{7}{4} S_4 + 4^3 \binom{7}{6} S_6 = 4^3 N^7 -N, \\
& \quad\quad\,\,\,\, \vdots \\
& 4 \binom{2k+1}{2} S_2 + 4^2 \binom{2k+1}{4} S_4 + \cdots + 4^{k-1} \binom{2k+1}{2k-2} S_{2k-2}
+ 4^k \binom{2k+1}{2k} S_{2k} = 4^k N^{2k+1} -N,
\end{aligned}
\end{equation*}
or, in matrix form,
\begin{equation*}
\begin{pmatrix}
4\binom{3}{2} & 0 & 0 & \! \hdots & 0 & 0 \\[3pt]
4\binom{5}{2} & \! 4^2 \binom{5}{4} & 0 & \! \hdots & 0 & 0 \\[3pt]
4\binom{7}{2} & \! 4^2 \binom{7}{4} & \! 4^3 \binom{7}{6} & \! \hdots  & 0 & 0 \\[3pt]
\vdots & \vdots & \ddots & \! \ddots & \vdots & \vdots \\[3pt]
4\binom{2k-1}{2} & \! 4^2 \binom{2k-1}{4} & \! 4^3 \binom{2k-1}{6} & \!\hdots & \! 4^{k-1}\binom{2k-1}{2k-2} & 0\\[5pt]
4\binom{2k+1}{2} & \! 4^2 \binom{2k+1}{4} & \! 4^3 \binom{2k+1}{6} & \! \hdots & \! 4^{k-1}\binom{2k+1}{2k-2} & \! 4^k \binom{2k+1}{2k}
\end{pmatrix}\!
\begin{pmatrix} S_2 \\[3pt] S_4 \\[3pt] S_6 \\ \vdots \\[3pt] S_{2k-2} \\[3pt] S_{2k}
\end{pmatrix}
= N \! \begin{pmatrix}
(2N)^2 -1 \\[4pt] (2N)^4 -1 \\[4pt] (2N)^6 -1 \\ \vdots \\[3pt] (2N)^{2k-2} -1 \\[3pt](2N)^{2k} -1
\end{pmatrix}.
\end{equation*}
Then, solving for $S_{2k}$, and applying Cramer's rule to the above triangular system of equations leads to the following determinantal formula for $S_{2k}$:
\begin{equation}\label{det1}
S_{2k} = \Delta_k \, N \!
\begin{vmatrix}
\binom{3}{2} & 0 & 0 & \! \hdots & 0 & \! (2N)^2 -1 \\[3pt]
\binom{5}{2} & \! \binom{5}{4} & 0 & \! \hdots & 0 & \! (2N)^4 -1 \\[3pt]
\binom{7}{2} & \! \binom{7}{4} & \! \binom{7}{6} & \! \hdots  & 0 & \! (2N)^6 -1 \\[3pt]
\vdots & \vdots & \vdots & \! \ddots & \vdots & \vdots \\[3pt]
\binom{2k-1}{2} & \! \binom{2k-1}{4} & \! \binom{2k-1}{6} & \!\hdots & \! \binom{2k-1}{2k-2} &
\! (2N)^{2k-2}-1 \\[5pt]
\binom{2k+1}{2} & \! \binom{2k+1}{4} & \! \binom{2k+1}{6} & \! \hdots & \! \binom{2k+1}{2k-2} &
\! (2N)^{2k} - 1
\end{vmatrix},
\end{equation}
where
\begin{equation*}
\Delta_ k = \frac{(k+1)!}{(2k+2)!\, 2^{k-1}},
\end{equation*}
and where, for $k=1$, $S_2 = \Delta_2 N [(2N)^2 -1]$. Formula \eqref{det1} gives $S_{2k}$ in the form of $N$ times an even polynomial in $N$ of degree $2k$ (or, equivalently, in the form of an odd polynomial in $N$ of degree $2k+1$). Indeed, developing the determinant on the right-hand side of  \eqref{det1} along its last column, we find that
\begin{equation*}
S_{2k} = \Delta_k N \left( \sum_{i=1}^k \alpha_{ik} (2N)^{2i} - \sum_{i=1}^k \alpha_{ik} \right),
\end{equation*}
where, for each $i=1,2,\ldots,k$, $\alpha_{ik}$ is the cofactor of the $i$th element of the last column of the determinant in \eqref{det1}. For example, for $k=5$, formula \eqref{det1} yields
\begin{align*}
S_{10} & = \frac{N}{10\,644\,480}
\begin{vmatrix}
3 & 0 & 0 & 0 & (2N)^2 -1 \\
10 & 5 & 0 & 0 & (2N)^4 -1 \\
21 & 35 & 7 & 0 & (2N)^6 -1 \\
36 & 126 & 84 & 9 & (2N)^8 -1 \\
55 & 330 & 463 & 165 & (2N)^{10} -1
\end{vmatrix}  \\[1mm]
& = \frac{1}{11}N^{11} - \frac{5}{12}N^9 + \frac{7}{8}N^7 - \frac{31}{32}N^5 + \frac{127}{256}N^3 - \frac{2555}{33792}N.
\end{align*}

Similarly, starting from \eqref{th2}, we can proceed as before to derive the following determinantal formula for $S_{2k-1}$:
\begin{equation}\label{det2}
S_{2k-1} = \Omega_k
\begin{vmatrix}
\binom{2}{1} & 0 & 0 & \! \hdots & 0 & \! (2N)^2 -1 \\[3pt]
\binom{4}{1} & \! \binom{4}{3} & 0 & \! \hdots & 0 & \! (2N)^4 -1 \\[3pt]
\binom{6}{1} & \! \binom{6}{3} & \! \binom{6}{5} & \! \hdots  & 0 & \! (2N)^6 -1 \\[3pt]
\vdots & \vdots & \vdots & \! \ddots & \vdots & \vdots \\[3pt]
\binom{2k-2}{1} & \! \binom{2k-2}{3} & \! \binom{2k-2}{5} & \!\hdots & \! \binom{2k-2}{2k-3} &
\! (2N)^{2k-2}-1\\[5pt]
\binom{2k}{1} & \! \binom{2k}{3} & \! \binom{2k}{5} & \!\hdots & \! \binom{2k}{2k-3} &
\! (2N)^{2k} -1
\end{vmatrix},
\end{equation}
where
\begin{equation*}
\Omega_k = \frac{1}{k!\, 8^k},
\end{equation*}
and where, for $k=1$, $S_1 = \Omega_1 [(2N)^2 -1]$. Now, the formula \eqref{det2} gives $S_{2k-1}$ in the form of an even polynomial in $N$ of degree $2k$. This can be seen by developing the determinant on the right-hand side of \eqref{det2} along its last column, as follows
\begin{equation*}
S_{2k-1} = \Omega_k \left( \sum_{i=1}^k \alpha_{ik}^{\prime} (2N)^{2i} - \sum_{i=1}^k \alpha_{ik}^{\prime} \right),
\end{equation*}
where, for each $i=1,2,\ldots,k$, $\alpha_{ik}^{\prime}$ is the cofactor of the $i$th element of the last column of the determinant in \eqref{det2}. For example, for $k=5$, formula \eqref{det2} yields
\begin{align*}
S_{9} & = \frac{1}{3\,932\,160}
\begin{vmatrix}
2 & 0 & 0 & 0 & (2N)^2 -1 \\
4 & 4 & 0 & 0 & (2N)^4 -1 \\
6 & 20 & 6 & 0 & (2N)^6 -1 \\
8 & 56 & 56 & 8 & (2N)^8 -1 \\
10 & 120 & 252 & 120 & (2N)^{10} -1
\end{vmatrix}  \\[1mm]
& = \frac{1}{10}N^{10} - \frac{3}{8}N^8 + \frac{49}{80}N^6 - \frac{31}{64}N^4 + \frac{381}{2560}N^2 - \frac{31}{2048}.
\end{align*}

On the other hand, by using the relation $(2N)^{2i} = (1+8S_1)^i$, and noting that $\Delta_ k N = (S_2/S_1) \Lambda_k$, where
\begin{equation*}
\Lambda_ k = \frac{3 \cdot (k+1)!}{(2k+2)!\, 2^{k}},
\end{equation*}
we can write \eqref{det1} in terms of $S_1$ and $S_2$ as
\begin{equation}\label{det3}
S_{2k} =  \left(\frac{S_2}{S_1}\right) \, \Lambda_k  \!
\begin{vmatrix}
\binom{3}{2} & 0 & 0 & \! \hdots & 0 & \! 8 S_1 \\[3pt]
\binom{5}{2} & \! \binom{5}{4} & 0 & \! \hdots & 0 & \! (1+8S_1)^2 -1 \\[3pt]
\binom{7}{2} & \! \binom{7}{4} & \! \binom{7}{6} & \! \hdots  & 0 & \! (1+8S_1)^3 -1 \\[3pt]
\vdots & \vdots & \vdots & \! \ddots & \vdots & \vdots \\[3pt]
\binom{2k-1}{2} & \! \binom{2k-1}{4} & \! \binom{2k-1}{6} & \!\hdots & \! \binom{2k-1}{2k-2} &
\! (1+8S_1)^{k-1} -1 \\[5pt]
\binom{2k+1}{2} & \! \binom{2k+1}{4} & \! \binom{2k+1}{6} & \! \hdots & \! \binom{2k+1}{2k-2} &
\! (1+8S_1)^{k} - 1
\end{vmatrix}.
\end{equation}
Similarly, we can write \eqref{det2} in terms of $S_1$ as
\begin{equation}\label{det4}
S_{2k-1} = \Omega_k
\begin{vmatrix}
\binom{2}{1} & 0 & 0 & \! \hdots & 0 & \! 8S_1 \\[3pt]
\binom{4}{1} & \! \binom{4}{3} & 0 & \! \hdots & 0 & \! (1+8S_1)^2 -1 \\[3pt]
\binom{6}{1} & \! \binom{6}{3} & \! \binom{6}{5} & \! \hdots  & 0 & \! (1+8S_1)^3 -1 \\[3pt]
\vdots & \vdots & \vdots & \! \ddots & \vdots & \vdots \\[3pt]
\binom{2k-2}{1} & \! \binom{2k-2}{3} & \! \binom{2k-2}{5} & \!\hdots & \! \binom{2k-2}{2k-3} &
\! (1+8S_1)^{k-1} -1\\[5pt]
\binom{2k}{1} & \! \binom{2k}{3} & \! \binom{2k}{5} & \!\hdots & \! \binom{2k}{2k-3} &
\! (1+8S_1)^{k} -1
\end{vmatrix}.
\end{equation}
The constant $\Lambda_k$ in front of the determinant in \eqref{det3} has been chosen so that the formula \eqref{det3} gives $S_{2k}$ directly as $S_2$ times a polynomial in $S_1$ of degree $k-1$. Furthermore, as it turns out, the formula \eqref{det4} gives $S_{2k-1}$ as $S_1^2$ times a polynomial in $S_1$ of degree $k-2$. Considering again the case of $k=5$, from \eqref{det3} and \eqref{det4} we obtain
\begin{align*}
S_{10} & = \left(\frac{S_2}{S_1}\right) \, \frac{1}{7\,096\,320}
\begin{vmatrix}
3 & 0 & 0 & 0 & 8S_1 \\
10 & 5 & 0 & 0 & (1+8S_1)^2 -1 \\
21 & 35 & 7 & 0 & (1+8S_1)^3 -1 \\
36 & 126 & 84 & 9 & (1+8S_1)^4 -1 \\
55 & 330 & 463 & 165 & (1+8S_1)^{5} -1
\end{vmatrix}  \\[2mm]
& \quad\quad\quad = S_2 \bigg( \frac{48}{11}S_1^4 - \frac{80}{11}S_1^3 + \frac{68}{11}S_1^2 - \frac{30}{11}S_1
+ \frac{5}{11} \bigg),
\end{align*}
and
\begin{align*}
S_{9} & = \frac{1}{3\,932\,160}
\begin{vmatrix}
2 & 0 & 0 & 0 & 8S_1 \\
4 & 4 & 0 & 0 & (1+8S_1)^2 -1 \\
6 & 20 & 6 & 0 & (1+8S_1)^3 -1 \\
8 & 56 & 56 & 8 & (1+8S_1)^4 -1 \\
10 & 120 & 252 & 120 & (1+8S_1)^{5} -1
\end{vmatrix}  \\[2mm]
& \quad\quad\quad\quad\quad = S_1^2 \bigg( \frac{16}{5}S_1^3 - 4 S_1^2 + \frac{12}{5}S_1 - \frac{3}{5} \bigg),
\end{align*}
respectively.

Naturally, the determinantal formulas \eqref{det1} and \eqref{det3} (respectively, \eqref{det2} and \eqref{det4}) are equivalent to each other, and constitute the two varieties of the Faulhaber theorem for the even-indexed power sums $S_{2k}$ (respectively, the odd-indexed power sums $S_{2k-1}$). For this reason, such formulas can be considered as the determinantal version of Faulhaber's theorem on sums of powers of integers \cite{beardon,knuth}

We can readily obtain a determinantal formula for the $2k$th Bernoulli number $B_{2k}$ starting from \eqref{det1}. To do this, we use the fact that, for all $k \geq 1$, the derivative of the even-indexed power sum $S_{2k}$ with respect to $n$ (considering $n$ as a continuous variable), when evaluated at $n=0$, is equal to $B_{2k}$ (see e.g. \cite{wu}). Hence, taking the derivative of the right-hand side of \eqref{det1} with respect to $n$ (recall that $N = n + \frac{1}{2}$), and making $n=0$, results in the following determinantal formula of order $k$ for $B_{2k}$:
\begin{equation}\label{detber}
B_{2k} = \frac{4 \cdot (k+1)!}{(2k+2)! \, 2^{k}} \!
\begin{vmatrix}
\binom{3}{2} & 0 & 0 & \! \hdots & 0 & \! 1 \\[3pt]
\binom{5}{2} & \! \binom{5}{4} & 0 & \! \hdots & 0 & \! 2 \\[3pt]
\binom{7}{2} & \! \binom{7}{4} & \! \binom{7}{6} & \! \hdots  & 0 & \! 3 \\[3pt]
\vdots & \vdots & \vdots & \! \ddots & \vdots & \vdots \\[3pt]
\binom{2k-1}{2} & \! \binom{2k-1}{4} & \! \binom{2k-1}{6} & \!\hdots & \! \binom{2k-1}{2k-2} &
\! k-1 \\[5pt]
\binom{2k+1}{2} & \! \binom{2k+1}{4} & \! \binom{2k+1}{6} & \! \hdots & \! \binom{2k+1}{2k-2} &
\! k
\end{vmatrix}.
\end{equation}
For example, for $k=6$, this formula yields
\begin{equation*}
B_{12} = \frac{1}{276\,756\,480} \,
\begin{vmatrix}
3 & 0 & 0 & 0 & 0 & 1 \\
10 & 5 & 0 & 0 & 0 & 2 \\
21 & 35 & 7 & 0 & 0 & 3 \\
36 & 126 & 84 & 9 & 0 & 4 \\
55 & 330 & 463 & 165 & 11 & 5 \\
78 & 715 & 1716 & 1287 & 286 & 6
\end{vmatrix}
= -\frac{691}{2730}.
\end{equation*}
Our formula \eqref{detber} for $B_{2k}$ may be compared with that obtained by Van Malderen in \cite{van}, namely
\begin{equation*}
B_{2k} = \frac{(-1)^{k+1}(2k)!}{2(2^{2k-1}-1)}
\begin{vmatrix}
\frac{1}{3!} & 1 & 0 & \! \hdots & \, 0 & 0 \\[3pt]
\frac{1}{5!} & \! \frac{1}{3!} & 1 & \! \hdots & \, 0 & 0 \\[3pt]
\frac{1}{7!} & \! \frac{1}{5!} & \! \frac{1}{3!} & \! \hdots  & \, 0 & 0 \\[3pt]
\vdots & \vdots & \vdots & \! \ddots & \vdots & \vdots \\[3pt]
\frac{1}{(2k-1)!} & \! \frac{1}{(2k-3)!} & \! \frac{1}{(2k-5)!} & \!\hdots & \, \frac{1}{3!} & 1 \\[5pt]
\frac{1}{(2k+1)!} & \! \frac{1}{(2k-1)!} & \! \frac{1}{(2k-3)!} & \! \hdots & \, \frac{1}{5!} & \frac{1}{3!}
\end{vmatrix}.
\end{equation*}
Other determinantal formulas for the Bernoulli numbers and polynomials can be found in \cite{chen,costa,qi}.

We conclude this section with the following important observation.

\begin{remark}
The Faulhaber form of $S_{2k-1}$, when expressed as a polynomial in $S_1$, is given by
\begin{equation*}
S_{2k-1} = c_{k,2} S_1^2 + c_{k,3} S_1^3 + c_{k,4} S_1^4 + \cdots + c_{k,k} S_1^{k}, \quad k \geq 2,
\end{equation*}
where all the $c_{k,2},c_{k,3},c_{k,4},\ldots\,$ are nonzero rational coefficients. As was pointed out elsewhere, the lack of the term in $S_1$ is due to the fact that, for all $k \geq 2$, the derivative of the polynomial $S_{2k-1}$ with respect to $n$ (considering $n$ as a continuous variable) is equal to zero when evaluated at $n=0$, which we may write as $S^{\prime}_{2k-1}\big|_{n=0} =0$. Indeed, upon looking at the determinantal formula \eqref{det4}, it is readily verified that, for $k \geq 2$,
\begin{equation*}
S^{\prime}_{2k-1}\big|_{n=0} = \Omega_k
\begin{vmatrix}
\binom{2}{1} & 0 & 0 & \! \hdots & 0 & \! 1 \cdot 4 \\[3pt]
\binom{4}{1} & \! \binom{4}{3} & 0 & \! \hdots & 0 & \! 2 \cdot 4 \\[3pt]
\binom{6}{1} & \! \binom{6}{3} & \! \binom{6}{5} & \! \hdots  & 0 & \! 3 \cdot 4 \\[3pt]
\vdots & \vdots & \ddots & \! \ddots & \vdots & \vdots \\[3pt]
\binom{2k-2}{1} & \! \binom{2k-2}{3} & \! \binom{2k-2}{5} & \!\hdots & \! \binom{2k-2}{2k-3} &
\! (k-1) \cdot 4 \\[5pt]
\binom{2k}{1} & \! \binom{2k}{3} & \! \binom{2k}{5} & \!\hdots & \! \binom{2k}{2k-3} &
\! k \cdot 4
\end{vmatrix}
=0,
\end{equation*}
the determinant above being equal to zero because its $k$th column vector is proportional to its first column vector.
\end{remark}

\section{A connection with the Chebyshev polynomials of the second kind}

In this section, we shall make use of a few fundamental properties of the Chebyshev polynomials of the second kind $U_n(x)$ (for these and many other properties of the Chebyshev polynomials, the reader may consult the Wikipedia entry \cite{wiki}). For $n \geq 0$, the Chebyshev polynomials of the second kind are defined recursively by $U_0(x) =1$, $U_1(x) = 2x$, and for $n \geq 2$, $U_n(x) = 2x U_{n-1}(x) -U_{n-2}(x)$. The first few polynomials $U_n(x)$ are given by
\begin{align*}
U_2(x) & = 4x^2 -1, & U_6(x) & = 64x^6 -80x^4 +24x^2 -1, \\
U_3(x) & = 8x^3 -4x, &  U_7(x) & = 128x^7-192x^5+80x^3 -8x, \\
U_4(x) & = 16x^4 -12x^2 +1, & U_8(x) & = 256x^8 -448x^6 +240x^4 -40x^2 +1, \\
U_5(x) & = 32x^5 -32x^3 +6x, & U_9(x) & = 512x^9 - 1024x^7 +672x^5 -160x^3 +10x,
\end{align*}
and they can be explicitly computed by
\begin{equation}\label{che1}
U_n(x) = \sum_{j=0}^{\lfloor n/2 \rfloor} (-1)^j \binom{n-j}{j} (2x)^{n-2j},
\end{equation}
where $\lfloor x \rfloor$ is the floor function. Another useful representation of $U_n$ is given by
\begin{equation*}
U_n(x) = \sum_{j=0}^{n} (-2)^j \binom{n+j+1}{2j+1} (1-x)^j,
\end{equation*}
from which we can easily deduce that
\begin{equation}\label{che2}
D_x^j U_n(x)\big|_{x =1} = 2^j j! \binom{n+j+1}{2j+1}, \quad j \geq 0.
\end{equation}
In particular, for $j=0$, we find that $U_n(1) = n+1$.

Of special interest for our purpose is the trigonometric definition of the Chebyshev polynomials of the second kind, namely
\begin{equation}\label{che3}
U_n( \cos x) = \frac{\sin((n+1)x)}{\sin x},
\end{equation}
or, equivalently,
\begin{equation}\label{che4}
U_{2n}( \cos (x/2)) = \frac{\sin Nx}{\sin (x/2)}.
\end{equation}
Incidentally, it is to be noted that, in the field of Fourier analysis, the collection of functions defined by $D_n(x) = \sin \big(n + \frac{1}{2}\big) x/\sin (x/2)$ is known as the Dirichlet kernel.

\subsection{Determining the power sums $S_{2k}$}

Combining \eqref{rpart} and \eqref{id1}, we have
\begin{equation*}
\sum_{k=0}^{\infty} (-1)^k S_{2k} \frac{x^{2k}}{(2k)!} = \frac{\sin Nx}{2\sin(x/2)} -\frac{1}{2}.
\end{equation*}
Therefore, it follows that, for $k \geq 1$,
\begin{equation}\label{met3}
S_{2k} = \frac{1}{2} (-1)^k \lim_{x\to 0} \, \frac{d^{2k}}{d x^{2k}} \left( \frac{\sin Nx}{\sin(x/2)} \right).
\end{equation}
For the simplest case of $k=1$, from \eqref{met3} we obtain
\begin{align*}
S_2 & = -\frac{1}{2} \lim_{x\to 0} \, \frac{d^{2}}{d x^{2}} \left( \frac{\sin Nx}{\sin(x/2)} \right) \\
& = \lim_{x\to 0} \,\left( \frac{N^2 \sin Nx}{2\sin(x/2)} + \frac{N \cos(x/2) \cos Nx}{2\sin^2 (x/2)}
- \frac{(3+\cos x)\sin Nx}{16\sin^3 (x/2)} \right) \\
& = N^3 + \lim_{x\to 0} \,\left( \frac{N \cos(x/2) \cos Nx}{2\sin^2 (x/2)}
- \frac{(3+\cos x)\sin Nx}{16\sin^3 (x/2)} \right).
\end{align*}
By using the \emph{Mathematica\small{\textsuperscript\circledR}} software, we are able to determine the first terms of the power series expansion about the point $x=0$ for the two functions within the last parenthesis. These are given by
\begin{align*}
\frac{N \cos(x/2) \cos Nx}{2\sin^2 (x/2)} & = 2N x^{-2} - \Big ( N^3 + \frac{N}{12} \Big) +
\frac{1}{960} \big( 80N^5 +40N^3 -7N \big) x^2 + O(x^4) + \ldots\, , \\[-2mm]
\intertext{and}
\frac{(3+\cos x)\sin Nx}{16\sin^3 (x/2)} & = 2N x^{-2} - \frac{N^3}{3} + \Big( \frac{N^5}{60} +
\frac{7N}{960} \Big) x^2 + O(x^4) + \ldots\, ,
\end{align*}
from which we conclude that
\begin{equation*}
S_2 = \frac{N^3}{3} - \frac{N}{12} = \frac{1}{6}n(n+1)(2n+1).
\end{equation*}
Needless to say, the process of calculating $S_{2k}$ by means of formula \eqref{met3} becomes excessively complex as $k$ increases, even more than what happened with the formulas \eqref{met1} and \eqref{met2}.

Interestingly, by virtue of \eqref{che4}, we can write \eqref{met3} in terms of the Chebyshev polynomials of the second kind as follows
\begin{equation}\label{met4}
S_{2k} = \frac{1}{2} (-1)^k D_x^{2k} \big( U_{2n}(\cos(x/2)) \big)\big|_{x=0}, \quad k \geq 1.
\end{equation}
It should be emphasized that the computation of $S_{2k}$ by means of \eqref{met4} and \eqref{che2} is more affordable to achieve than by directly utilizing formula \eqref{met3}.

\begin{example}
Let us denote $t \equiv \cos(x/2)$. Then, for $k=1$, from \eqref{met4} we obtain
\begin{align*}
S_2 & = -\frac{1}{2} D_x^2 \big( U_{2n}(\cos(x/2)) \big)\big|_{x=0} \\
& = \frac{1}{8} D_t \big( U_{2n}(t) \big)\big|_{t=1} = \frac{1}{4} \binom{2n+2}{3} = \frac{1}{6}n(n+1)(2n+1),
\end{align*}
where we have used \eqref{che2} to get $D_t \big( U_{2n}(t) \big)\big|_{t=1}$. Similarly, for $k=2$, from \eqref{met4} and \eqref{che2} we obtain
\begin{align*}
S_4 & = \frac{1}{2} D_x^4 \big( U_{2n}(\cos(x/2)) \big)\big|_{x=0} \\
& = \frac{1}{32} D_t \big( U_{2n}(t) \big)\big|_{t=1} + \frac{3}{32} D_t^2 \big( U_{2n}(t) \big)\big|_{t=1} \\
& = \frac{1}{16} \binom{2n+2}{3} + \frac{3}{4} \binom{2n+3}{5} = \frac{1}{30}n(n+1)(2n+1)(3n^2 + 3n-1).
\end{align*}
\end{example}

Now, to systematically evaluate the higher order derivative on the right-hand side of \eqref{met4}, $D_x^{2k} \big( U_{2n}(\cos(x/2)) \big)\big|_{x=0}$, we invoke the well-known Fa\`{a} di Bruno's formula \cite{johnson} for the $k$th derivative of a composite function. In doing so, and using \eqref{che2} and the fact that $D_x^{2j+1} \cos(x/2)\big|_{x=0} =0$ for all $j \geq 0$, we arrive straightforwardly at the following result.

\begin{theorem}\label{th:2}
For any integer $k \geq 1$, the even-indexed power sum $S_{2k}$ is given by
\begin{equation}\label{bruno}
S_{2k} = \frac{1}{2} \sum \frac{(2k)!}{b_1! b_2! \cdots b_k!} \prod_{r=1}^{k} \left( \frac{1}{4^r (2r)!}
\right)^{b_r} 2^m m! \binom{2n+m+1}{2m+1},
\end{equation}
where the sum is over all $k$-tuples of nonnegative integers $(b_1,b_2,\ldots,b_k)$ satisfying the constraint $b_1 + 2b_2 + \cdots + k b_k = k$, and where $m = b_1 + b_2 +\cdots +b_k$.
\end{theorem}

\begin{example}
For $k=3$, the nonnegative solutions $(b_1,b_2,b_3)$ of the equation $b_1 +2b_2 +3b_3 =3$ are $(0,0,1)$, $(1,1,0)$, and $(3,0,0)$. Hence, using these solutions into formula \eqref{bruno} yields
\begin{align*}
S_6 & = \frac{1}{64} \binom{2n+2}{3} + \frac{15}{16} \binom{2n+3}{5} + \frac{45}{8}\binom{2n+4}{7} \\
& = \frac{1}{42}n(n+1)(2n+1)(3n^4 +6n^3 -3n+1).
\end{align*}
\end{example}

Notice that when the sum runs over all possible integer partitions of $k$, the parameter $m$ on the right-hand side of \eqref{bruno} takes on all integer values within the interval $1 \leq m \leq k$. Consequently, formula \eqref{bruno} provides $S_{2k}$ in the polynomial form
\begin{equation*}
S_{2k} = \sum_{m=1}^k p_{k,m} \binom{2n+m+1}{2m+1}, \quad k \geq 1,
\end{equation*}
where all the $p_{k,1}, p_{k,2}, \ldots, p_{k,k}$ are nonzero rational coefficients independent of $n$, with $p_{k,k} = (2k)!/2^{2k+1}$.

Alternatively, one can compute $S_{2k}$ by using the explicit formula \eqref{che1} for $U_n(x)$ into \eqref{met4}. This gives
\begin{equation*}
S_{2k} = \frac{1}{2} (-1)^k   \sum_{j=0}^{n} (-1)^j \binom{2n-j}{j} D_x^{2k}
\Big( \big(2\cos(x/2) \big)^{2n-2j} \Big)\Big|_{x=0},
\end{equation*}
or, equivalently,
\begin{equation}\label{met5}
S_{2k} = \frac{1}{2} (-1)^k \sum_{j=1}^{n} (-1)^{n+j} 4^j \binom{n+j}{2j}
D_x^{2k} \Big( \big(\cos(x/2) \big)^{2j} \Big)\Big|_{x=0},
\end{equation}
where the sum on the right-hand side of \eqref{met5} starts at $j=1$ because the involved derivative is obviously equal to zero when $j=0$. For $k=1,2,3$, and $4$, from \eqref{met5} we obtain the rather exotic formulas
\begin{align*}
S_2 & = \frac{1}{4} \sum_{j=1}^n (-1)^{n+j} 4^j j \binom{n+j}{2j}, \\
S_4 & = \frac{1}{8} \sum_{j=1}^n (-1)^{n+j} 4^j j(3j -1) \binom{n+j}{2j}, \\
S_6 & = \frac{1}{16} \sum_{j=1}^n (-1)^{n+j} 4^j j \big(15j^2 -15j +4 \big) \binom{n+j}{2j}, \\[-2mm]
\intertext{and}
S_8 & = \frac{1}{32} \sum_{j=1}^n (-1)^{n+j} 4^j j \big(105j^3 -210j^2 + 147j -34 \big) \binom{n+j}{2j},
\end{align*}
respectively. In general, $S_{2k}$ can be expressed in the form
\begin{equation}\label{ansaz}
S_{2k} = \frac{1}{2^{k+1}} \sum_{j=1}^n (-1)^{n+j} 4^j P_k(j) \binom{n+j}{2j}, \quad k \geq 1,
\end{equation}
where
\begin{equation*}
P_k(j) = (-1)^k 2^k D_x^{2k} \Big( \big(\cos(x/2) \big)^{2j} \Big)\Big|_{x=0}.
\end{equation*}
As it turns out, $P_k(j)$ is a polynomial in $j$ of degree $k$ without constant term and leading coefficient $(2k-1)!!$. Furthermore, the coefficients of $P_k(j)$ sum up to $2^{k-1}$ and have apparently alternating signs (the leading coefficient being positive). Indeed, by using Fa\`{a} di Bruno's formula to evaluate $D_x^{2k} \big( \big(\cos(x/2) \big)^{2j} \big)\big|_{x=0}$, and recalling that $D_x^{2j+1} \cos(x/2)\big|_{x=0} =0$ for all $j \geq 0$, it can be shown that
\begin{equation}\label{bruno2}
P_k(j) = 2^k \sum \frac{(2k)!}{b_1! b_2! \cdots b_k!} \prod_{r=1}^{k} \left( \frac{1}{4^r (2r)!}
\right)^{b_r}(2j)_m, \quad k \geq 1,
\end{equation}
where $(x)_m$ denotes the falling factorial $x(x-1)\cdots(x-m+1)$. As before, the sum in \eqref{bruno2} is over all solutions in nonnegative integers $b_1,b_2,\ldots,b_k$ of the equation $b_1 + 2b_2 + \cdots + k b_k = k$, with $m = b_1 + b_2 +\cdots +b_k$. For example, when $k=5$, the nonnegative solutions to the equation $b_1 +2b_2 +3b_3 +4b_4 +5b_5 =5$ are $(0,0,0,0,1)$, $(1,0,0,1,0)$, $(0,1,1,0,0)$, $(1,2,0,0,0)$, $(2,0,1,0,0)$, $(3,1,0,0,0)$, and $(5,0,0,0,0)$. Using these solutions into formula \eqref{bruno2}, and after some simple algebra, we find that
\begin{equation*}
P_{5}(j) = j \big( 945j^4 - 3150j^3 + 4095j^2 - 2370j + 496).
\end{equation*}
According to \eqref{ansaz}, we then have
\begin{equation*}
S_{10} = \frac{1}{64} \sum_{j=1}^n (-1)^{n+j} 4^j j \big( 945j^4 - 3150j^3 + 4095j^2 - 2370j + 496)
\binom{n+j}{2j}.
\end{equation*}

\begin{remark}
The preceding formulas for $S_4$, $S_6$, $S_8$, and $S_{10}$, can be rewritten as
\begin{align*}
S_4 & = \frac{3}{8} \sum_{j=1}^{n} (-1)^{n+j} 4^j j^2 \binom{n+j}{2j} - \frac{1}{2}S_2, \\
S_6 & = \frac{15}{16} \sum_{j=1}^{n} (-1)^{n+j} 4^j j^3 \binom{n+j}{2j} - \frac{1}{4}S_2 - \frac{5}{2}S_4, \\
S_8 & = \frac{105}{32} \sum_{j=1}^{n} (-1)^{n+j} 4^j j^4 \binom{n+j}{2j} +\frac{1}{8}S_2 - \frac{21}{4}S_4 - 7S_6, \\[-2mm]
\intertext{and}
S_{10} & = \frac{945}{64} \sum_{j=1}^{n} (-1)^{n+j} 4^j j^5 \binom{n+j}{2j} -\frac{33}{2}S_2 + \frac{163}{16}S_4
-\frac{147}{4} S_6 -15S_8,
\end{align*}
respectively. In general, from these expressions we can deduce the recurrence relation
\begin{equation*}
\sum_{r=1}^{k} c_{k,r} S_{2r} = \sum_{j=1}^{n} (-1)^{n+j} 4^j j^k \binom{n+j}{2j}, \quad k \geq 1,
\end{equation*}
which holds for certain numerical coefficients $c_{k,r}$, $r=1,2,\ldots,k$, having the property $\sum_{r=1}^k c_{k,r} =4$. In particular, $c_{1,1} =4$.
\end{remark}

\subsection{Determining the power sums $S_{2k-1}$}

Analogously, regarding the odd-indexed power sums $S_{2k-1}$, it follows from \eqref{ipart} and \eqref{id2} that
\begin{equation}\label{odd}
\sum_{k=0}^{\infty} (-1)^k S_{2k+1} \frac{x^{2k+1}}{(2k+1)!} = \frac{\cos(x/2) - \cos Nx}{2\sin(x/2)}.
\end{equation}
Then, noting the trigonometric identity
\begin{equation*}
\cos(x/2) - \cos Nx = 2\sin((n+1)x/2) \sin(nx/2),
\end{equation*}
and recalling \eqref{che3}, we can write the right-hand side of \eqref{odd} as:
\begin{equation*}
\frac{\cos(x/2) - \cos Nx}{2\sin(x/2)} = U_{n}(\cos(x/2)) \sin(nx/2).
\end{equation*}
This allows us to express $S_{2k-1}$ in the form
\begin{equation}\label{met6}
S_{2k-1} = (-1)^{k-1} D_x^{2k-1} \big( U_{n}(\cos(x/2)) \sin(nx/2) \big)\big|_{x=0}, \quad k \geq 1.
\end{equation}
Now, applying Leibniz's formula to evaluate the higher order derivative of the product function\linebreak $U_{n}(\cos(x/2)) \sin(nx/2)$, and using the relations \eqref{pf1}, we obtain
\begin{align*}
S_{2k-1} & = (-1)^{k-1} \sum_{j=0}^{2k-1} \binom{2k-1}{j} D_x^{2k-1-j} \big( U_{n}(\cos(x/2)) \big)\big|_{x=0}
\, D_x^j \sin(nx/2)\big|_{x=0} \\
& = 2 (-1)^{k} \sum_{j=1}^{k} (-1)^j 4^{-j} \binom{2k-1}{2j-1} D_x^{2(k-j)} \big( U_{n}(\cos(x/2)) \big)\big|_{x=0}
\, n^{2j-1},
\end{align*}
or, equivalently,
\begin{equation}\label{met7}
S_{2k-1} = \frac{2}{4^{k}} \sum_{j=0}^{k-1} (-1)^j 4^{j} \binom{2k-1}{2j} D_x^{2j} \big( U_{n}(\cos(x/2)) \big)\big|_{x=0}
\, n^{2k-2j-1}, \quad k \geq 1.
\end{equation}
For $k=1$, from \eqref{met7} we quickly obtain
\begin{equation*}
S_1 = \frac{1}{2}n U_n(1) = \frac{1}{2}n(n+1).
\end{equation*}
Furthermore, for $k=2$, and letting $t \equiv \cos(x/2)$, from \eqref{met7} we obtain
\begin{align*}
S_3 & = \frac{1}{8}\Big[ (n+1)n^3 -12n D_x^{2} \big( U_{n}(\cos(x/2)) \big)\big|_{x=0} \Big] \\
& = \frac{1}{8}\Big[ (n+1)n^3 + 3n D_t \big( U_{n}(t) \big)\big|_{t=1} \Big] \\
& =  \frac{1}{8}\bigg[ (n+1)n^3 + 6n \binom{n+2}{3} \bigg] = \frac{1}{4}n^2 (n+1)^2,
\end{align*}
where we have used \eqref{che2} to get $D_t \big( U_{n}(t) \big)\big|_{t=1}$.

As was done in the case of the even-indexed power sums, we may evaluate $D_x^{2j} \big( U_{n}(\cos(x/2)) \big)\big|_{x=0}$ by means of  Fa\`{a} di Bruno's formula, and then plugging the resulting expression into \eqref{met7}. This allows us to establish the following result, which is the analogue of Theorem \ref{th:2} for the odd-indexed power sums.
\begin{theorem}\label{th:3}
For any integer $k \geq 1$, the odd-indexed power sum $S_{2k-1}$ is given by
\begin{align}\label{bruno3}
\!\! S_{2k-1} = \frac{2}{4^{k}} \Bigg[ (n+1) n^{2k-1} + \sum_{j=1}^{k-1} 4^j n^{2k-2j-1} & \binom{2k-1}{2j}
\sum \frac{(2j)!}{b_1! b_2! \cdots b_j!} \notag \\
& \left. \times \, \prod_{r=1}^{j} \left( \frac{1}{4^r (2r)!} \right)^{b_r} 2^m m! \binom{n+m+1}{2m+1} \right],
\end{align}
where $S_1 = \frac{1}{2}n(n+1)$, and where, for each $j =1,2,\ldots,k-1$, the rightmost sum is over all \mbox{$j$-tuples} of nonnegative integers $(b_1,b_2,\ldots,b_j)$ satisfying the constraint $b_1 + 2b_2 + \cdots + j b_j = j$, with $m = b_1 + b_2 +\cdots +b_j$.
\end{theorem}

\begin{example}
For $k=3$, formula \eqref{bruno3} yields
\begin{align*}
S_5 & = \frac{1}{32} \bigg[ (n+1)n^5 + (10n + 20n^3) \binom{n+2}{3} + 120n \binom{n+3}{5} \bigg] \\
& = \frac{1}{12} n^2 (n+1)^2 ( 2n^2 +2n -1).
\end{align*}
Likewise, for $k=4$, it gives
\begin{align*}
S_7 & = \frac{1}{128} \bigg[ (n+1)n^7 + (14n + 70n^3 +42n^5) \binom{n+2}{3}  \\
&  \quad\quad\quad\quad\,\,\,  + (840n +840n^3) \binom{n+3}{5} + 5040n \binom{n+4}{7} \bigg] \\
& = \frac{1}{24} n^2 (n+1)^2 (3n^4 + 6n^3 -n^2 -4n +2).
\end{align*}
\end{example}

In view of the expressions for $S_3$, $S_5$, and $S_7$ above, we can guess the general form of the power sum $S_{2k-1}$ obtained from \eqref{bruno3} to be
\begin{equation*}
S_{2k-1} = \frac{2}{4^{k}} \sum_{j=1}^k Q_{k,j}(n) \binom{n+j}{2j-1}, \quad k \geq 1,
\end{equation*}
where $Q_{k,j}(n)$ is an odd polynomial in $n$ of degree $2(k-j)+1$. Clearly, this formula holds for $k=1$ by setting $Q_{1,1}(n) =n$. In general, it turns out that $Q_{k,1}(n) = n^{2k-1}$ and $Q_{k,k} = (2k-1)! n$.

We end this section with the following remark.

\begin{remark}
Let us denote by $S_k(2n)$ the sum of the $k$th powers of the first $2n$ positive integers $1^k + 2^k + \cdots + (2n)^k$, to be distinguished from the sum of the $k$th powers of the first $n$ positive integers $S_k = 1^k + 2^k + \cdots + n^k$. Thus, using \eqref{met4} and \eqref{met7}, we can express $S_{2k-1}(2n)$ in the form
\begin{equation}\label{doub}
S_{2k-1}(2n) = (2n+1) n^{2k-1} + 2 \sum_{j=1}^{k-1} \binom{2k-1}{2j} S_{2j} \, n^{2k-2j-1}, \quad k \geq 1,
\end{equation}
with $S_1(2n) = (2n+1)n$. For example, putting $k=5$ in \eqref{doub} gives
\begin{equation*}
S_9(2n) = (2n+1)n^9 + 72n^7 S_2 + 252n^5 S_4 + 168n^3 S_6 + 18n S_8.
\end{equation*}
On the other hand, according to \cite[Equation (0.2)]{snow}, $S_{2k-1}(2n)$ can equally be expressed as
\begin{equation}\label{doub2}
S_{2k-1}(2n) = 2S_{2k-1} + n^{2k} + \sum_{j=1}^{2k-2} \binom{2k-1}{j} S_{j} \, n^{2k-1-j}, \quad k \geq 1.
\end{equation}
As a result, by combining \eqref{doub} and \eqref{doub2}, we can get yet another representation for $S_{2k-1}(2n)$, namely
\begin{equation*}
S_{2k-1}(2n) = 4S_{2k-1} - n^{2k-1} + 2\sum_{j=1}^{k-1} \binom{2k-1}{2j-1} S_{2j-1} \, n^{2k-2j}, \quad k \geq 1.
\end{equation*}
\end{remark}

\section{Concluding comments}

To summarize, in this paper we have first examined the standard differentiation formulas \eqref{met1} and \eqref{met2}. These formulas are intended to determine, for any given $k$, the power sum $S_k = 1^k + 2^k +\cdots + n^k$ as an explicit polynomial in $n$. As remarked earlier, however, using either \eqref{met1} or \eqref{met2} to get $S_k$ is very troublesome when one tries to calculate $S_k$ beyond the first values of $k$. The main point raised in this paper is that, in spite of this, we can still extract a lot of information about $S_k$ by exploiting the differentiation formula \eqref{met1}. Indeed, starting from \eqref{der1} and \eqref{der2}, we have derived a couple of recurrence relations involving the power sums $\{ S_{2j} \}_{j=1}^k$ and $\{ S_{2j-1} \}_{j=1}^k$, respectively, with both recurrences depending explicitly on the parameter $N = n + \frac{1}{2}$. This has allowed us to obtain, using Cramer's rule, a determinantal formula for $S_{2k}$ and $S_{2k-1}$ in terms of $N$ or, alternatively, in terms of $S_1$. As we have seen, our determinantal formulas for $S_{2k}$ and $S_{2k-1}$ represent the determinantal version of Faulhaber's theorem on sums of powers of integers. Furthermore, we have shown that $S_{2k}$ and $S_{2k-1}$ can be obtained by taking the corresponding higher order derivatives of the Chebyshev polynomials of the second kind, thus revealing a remarkable connection between these polynomials and the sums of powers of integers. We leave as an open question whether there is some kind of relationship between the power sums and the Chebyshev polynomials of the first kind.

The results obtained in the this paper can hopefully be generalized to the sum of the $k$th powers of the first $n$ terms of an arbitrary arithmetic progression with initial term $a$ and common difference $d$:
\begin{equation*}
S_k^{a,d} = a^k + (a+d)^k + (a+2d)^k + \cdots + ( a+(n-1)d)^k,
\end{equation*}
where $k,a,d,$ and $n$ are assumed to be integer variables with $k,a \geq 0$ and $d,n \geq 1$. Indeed, the exponential generating function of the sequence $S_0^{a,d}, S_1^{a,d},\ldots, S_k^{a,d}, \ldots, \, $ is given by
\begin{equation}\label{egf2}
\sum_{k=0}^{\infty} S_k^{a,d} \, \frac{x^k}{k!} = \sum_{r=1}^{n} e^{(a+ (r-1)d )x},
\end{equation}
which reduces to \eqref{egf} when $a = d=1$. Thus, from \eqref{egf2}, it follows that
\begin{equation}\label{met8}
S_k^{a,d} = \lim_{x\to 0} \, \frac{d^k}{d x^k} \left( \frac{e^{(a+nd)x} -e^{ax}}{e^{dx} -1} \right), \quad k \geq 0.
\end{equation}
As an example, for $k=1$, from \eqref{met8} one gets
\begin{equation*}
S_1^{a,d} = \lim_{x\to 0} \, \frac{d}{d x} \left( \frac{e^{(a+nd)x} -e^{ax}}{e^{dx} -1} \right)
= \frac{1}{2}n \big[ a + (a + (n-1)d)\big],
\end{equation*}
retrieving the well-known formula for $S_1^{a,d} = a + (a+d) + (a+2d) + \cdots + ( a+(n-1)d)$ as being $n$ times the arithmetic mean of the first and last terms. Let us observe that $S_1^{a,d}$ can be rewritten somewhat artificially in the form
\begin{equation*}
S_1^{a,d} = c_{1,0}^{a,d} + c_{1,1}^{a,d} \big( N_{a,d} \big)^2,
\end{equation*}
where $c_{1,0}^{a,d} = \frac{a}{2} -\frac{a^2}{2d} - \frac{d}{8}$, $c_{1,1}^{a,d} = \frac{d}{2}$, and $N_{a,d} = n + \frac{a}{d} - \frac{1}{2}$. This formula for $S_1^{a,d}$ constitutes a particular case of the general expression (cf. \cite[Theorem 2]{bazso})
\begin{equation*}
S_{2k-1}^{a,d} = \sum_{j=0}^{k} c_{k,j}^{a,d} \big( N_{a,d}\big)^{2j}, \quad k \geq 1.
\end{equation*}
Explicit formulas for the coefficients $c_{k,0}^{a,d}, c_{k,1}^{a,d}, \ldots, c_{k,k}^{a,d}$ can be found elsewhere. Incidentally, for the special case in which $a=1$ and $d=2$, the above general expression becomes
\begin{equation*}
1^{2k-1} + 3^{2k-1} + 5^{2k-1} + \cdots + (2n-1)^{2k-1} = \sum_{r=1}^{k} d_{k,r} n^{2r}, \quad k \geq 1.
\end{equation*}

On the other hand, by differentiating the identity
\begin{equation*}
\big( e^{dx} -1 \big) \sum_{r=1}^{n} e^{(a+ (r-1)d )x} = e^{(a +nd)x} - e^{ax}
\end{equation*}
$k$ times, and setting $x=0$, Wiener \cite{wiener} derived certain recurrences for the power sums $S_k^{a,d}$. Further recurrences for $S_k^{a,d}$ were obtained by Howard \cite{howard} by starting from the generating function \eqref{egf2}.

Furthermore, by making the transformations $e^{(a+nd)x} \to x^{a+nd}$, $e^{ax} \to x^{a}$, $e^{dx} \to x^{d}$,
$\frac{d^k}{d x^k} \to \big( x\frac{d}{dx} \big)^k$, and $\lim_{x\to 0} \to \lim_{x\to 1}$, the formula \eqref{met8} for $S_k^{a,d}$ can be written in the equivalent form
\begin{equation*}
S_k^{a,d} = \lim_{x\to 1} \, \left( x\frac{d}{d x}\right)^k \left( \frac{x^{a +nd} -x^a}{x^d -1} \right), \quad k \geq 0.
\end{equation*}
The method of obtaining $S_k^{a,d}$ based on this last formula was pursued by Gauthier in \cite{gaut2,gaut3}.

Finally, it is worth pointing out that, by using Equation (50) of \cite{xiong}, it can be deduced that
\begin{equation}\label{met9}
S_k^{a,d} = \frac{1}{(k+1)!} \, \lim_{x\to 1} \frac{d^{k+1}}{d x^{k+1}} Q_n^{a,d}(x;k),    \quad k \geq 0,
\end{equation}
where $Q_n^{a,d}(x;k)$ is a polynomial in $x$ of degree $n+k+1$ given by
\begin{equation*}
Q_n^{a,d}(x;k) = a^k x(x-1)^{k+1} + x^{n+1}T_k(x,a+d(n-1),-d) - x^2 T_k(x,a,-d).
\end{equation*}
In turn, $T_k(x,a,d)$ is a polynomial in $x$ of degree $k$ given by
\begin{equation*}
T_k(x,a,d) = \sum_{j=0}^k A_{k,j}(a,d) x^j,
\end{equation*}
where the involved coefficients $A_{k,j}(a,d)$, $j=0,1,\ldots,k$, are
\begin{equation*}
A_{k,j}(a,d) = \sum_{i=0}^j (-1)^i \big[(j+1-i)d -a \big]^k \binom{k+1}{i}.
\end{equation*}
(Please note that the definition of $T_k(x,a,d)$ given here is slightly different from the definition given in \cite{xiong}.) As an example, for $k=3$, we have that
\begin{align*}
Q_n^{a,d}(x;3) & = a^3 x(x-1)^4 + x^{n+1} \big[ (-a-d-d(n-1))^3 + \big( (-a-2d-d(n-1))^3 \\
& \quad -4(-a-d-d(n-1))^3 \big) x + \big( (-a-3d-d(n-1))^3 -4 (-a-2d-d(n-1))^3 \\
& \quad +6(-a-d-d(n-1))^3 \big) x^2 + \big( (-a-4d-d(n-1))^3 -4 (-a-3d-d(n-1))^3 \\
& \quad +6(-a-2d-d(n-1))^3 -4( -a-d-d(n-1))^3 \big) x^3 \big] - x^2 \big[ (-a-d)^3 \\
& \quad + \big( (-a-2d)^3 -4(-a-d)^3 \big) x + \big( (-a-3d)^3 -4 (-a-2d)^3 +6(-a-d)^3 \big) x^2 \\
& \quad + \big( (-a-4d)^3 -4 (-a-3d)^3 +6(-a-2d)^3 -4( -a-d)^3 \big) x^3 \big].
\end{align*}
Then, with the help of the \emph{Mathematica\small{\textsuperscript\circledR}} software, from \eqref{met9} we get
\begin{align*}
S_3^{a,d} & = \frac{1}{4} (4a^3 -6a^2 d + 2a d^2)n + \frac{1}{4} (6a^2 d -6a d^2 +d^3) n^2 \\
& \quad + \frac{1}{4} (4a d^2 - 2d^3) n^3 + \frac{1}{4} d^3 n^4.
\end{align*}
As expected, when $a=d=1$, we find that $S_3^{1,1} = \frac{1}{4} (n^2 + 2n^3 + n^4) = \frac{1}{4}n^2 (n+1)^2$.

\begin{remark}
Upon examining the proof of Lemma 14 of \cite{xiong}, we have noticed that one term is missing in the right-hand side of each of  Equations (49) and (50). Specifically, the missing term in both of them is $a^k x$.
\end{remark}

\end{document}